\documentclass[11pt]{article}
\usepackage{amssymb}
\usepackage{amsmath}
\usepackage{graphicx}
\usepackage{setspace}
\begin{document}

\newtheorem{theorem}{Theorem}[section]
\newtheorem{conjecture}[theorem]{Conjecture}
\newtheorem{corollary}[theorem]{Corollary}
\newtheorem{lemma}[theorem]{Lemma}
\newtheorem{claim}[theorem]{Claim}
\newtheorem{proposition}[theorem]{Proposition}
\newtheorem{construction}[theorem]{Construction}
\newtheorem{definition}[theorem]{Definition}
\newtheorem{question}[theorem]{Question}
\newtheorem{problem}[theorem]{Problem}
\newtheorem{remark}[theorem]{Remark}
\newtheorem{observation}[theorem]{Observation}

\newcommand*{\qed}{\hfill\ensuremath{\blacksquare}}%
\newcommand{\ex}{{\mathrm{ex}}}

\newcommand{\EX}{{\mathrm{EX}}}

\newcommand{\AR}{{\mathrm{AR}}}

\def\endproofbox{\hskip 1.3em\hfill\rule{6pt}{6pt}}
\newenvironment{proof}%
{%
\noindent{\it Proof.}
}%
{%
 \quad\hfill\endproofbox\vspace*{2ex}
}
\def\qed{\hskip 1.3em\hfill\rule{6pt}{6pt}}
\def\ce#1{\lceil #1 \rceil}
\def\fl#1{\lfloor #1 \rfloor}
\def\lr{\longrightarrow}
\def\e{\varepsilon}
\def\ex{{\rm\bf ex}}
\def\cA{{\cal A}}
\def\cB{{\cal B}}
\def\cC{{\cal C}}
    \def\cD{{\cal D}}
\def\cF{{\cal F}}
\def\cG{{\cal G}}
\def\cH{{\cal H}}
\def\ck{{\cal K}}
\def\cI{{\cal I}}
\def\cJ{{\cal J}}
\def\cL{{\cal L}}
\def\cM{{\cal M}}
\def\cP{{\cal P}}
\def\cQ{{\cal Q}}
\def\cR{{\cal R}}
\def\cS{{\cal S}}
\def\cE{{\cal E}}
\def\cT{{\cal T}}
\def\ex{{\rm ex}}
\def\pr{{\rm Pr}}
\def\exp{{\rm  exp}}

\def\wt{\widetilde{T}}
\def\bkl{{\cal B}^{(k)}_\ell}
\def\cmkt{{\cal M}^{(k)}_{t+1}}
\def\cpkl{{\cal P}^{(k)}_\ell}
\def\cckl{{\cal C}^{(k)}_\ell}
\def\pkl{\mathbb{P}^{(k)}_\ell}
\def\ckl{\mathbb{C}^{(k)}_\ell}

\def\mC{{\cal C}}

\def\imp{\Longrightarrow}
\def\1e{\frac{1}{\e}\log \frac{1}{\e}}
\def\ne{n^{\e}}
\def\rad{ {\rm \, rad}}
\def\equ{\Longleftrightarrow}
\def\pkl{\mathbb{P}^{(k)}_\ell}

\def\mE{\mathbb{E}}

\def\mP{\mathbb{P}}

\def \e{\epsilon}
\voffset=-0.5in
	
\setstretch{1.1}
\pagestyle{myheadings}
\markright{{\small \sc  Jiang, Ma, Yepremyan:}
  {\it\small On Tur\'an exponents of bipartite graphs}}

\title{\huge\bf  On Tur\'an exponents of bipartite graphs}

\author{
Tao Jiang\thanks{Department of Mathematics, Miami University, Oxford,
OH 45056, USA. E-mail: jiangt@miamioh.edu. Research supported in part
by National Science Foundation grant DMS-1400249. }
\quad \quad Jie Ma \thanks{
School of Mathematical Sciences,
University of Science and Technology of China, Hefei, 230026,
P.R. China. Email: jiema@ustc.edu.cn.
Research supported in part by National Natural Science Foundation of China grants 11501539 and 11622110.
}
\quad \quad Liana Yepremyan
\thanks{
Department of Mathematics, University of Oxford, UK. E-mail:yepremyan@maths.ox.ac.uk.
Research supported in part by ERC Consolidator Grant 647678.
    %{\rm\small {\jobname}.tex,}\newline ${}$ \hfill}
} }

\date{\today}

\maketitle
\begin{abstract}

A long-standing conjecture of  Erd\H{o}s and Simonovits asserts that for every rational number $r\in (1,2)$ there exists a  bipartite graph $H$ such that $\ex(n,H)=\Theta(n^r)$. So far this conjecture is known to be true only for rationals of form $1+1/k$ and $2-1/k$, for integers $k\geq 2$.
In this  paper we add a new form of rationals for which the conjecture is true; $2-2/(2k+1)$, for $k\geq 2$. This in its turn also gives an affirmative answer to a question of Pinchasi and Sharir on cube-like graphs.

Recently, a version of Erd\H{o}s and Simonovits's conjecture where one replaces a single graph by a family, was  confirmed by Bukh and Conlon. They proposed a construction of bipartite graphs which should satisfy  Erd\H{o}s and Simonovits's conjecture. Our result can also be viewed  as a first step towards verifying Bukh and Conlon's conjecture. 

We also prove the an upper bound on the Tur\'an's number of $\theta$-graphs in an asymmetric setting and employ this result to obtain yet another new rational exponent for Tur\'an exponents; $r=7/5$.
\end{abstract}

\section{Introduction}
Given a family $\cH$ of graphs,  a graph $G$ is called {\it $\cH$-free} if it contains no member of $\cH$ as a subgraph.
The Tur\'an number $ex(n,\cH)$ of $\cH$ is the maximum number of edges in an $n$-vertex $\cH$-free graph.
When $\cH$ consists of a single graph $H$, we
write $ex(n,H)$ for $ex(n,\{H\})$. The study of Tur\'an numbers plays a central
role in extremal graph theory.
The celebrated Erd\H{o}s-Simonovits-Stone theorem \cite{ES, EStone} states that
if $\chi(\cH)$ denotes the minimum chromatic number of a graph in $\cH$, then
$$\ex(n,\cH)=\left(1-\frac{1}{\chi(\cH)-1}\right)\binom{n}{2}+o(n^2).$$
Thus, the function is asymptotically determined if $\chi(\cH)\geq 3$.
If $\chi(\cH)=2$, that is, if $\cH$ contains a bipartite graph, then
this only gives $\ex(n,\cH)=o(n^2)$. We will refer to $\ex(n,\cH)$ with $\chi(\cH)=2$ as {\it degenerate Tur\'an numbers},
as in \cite{FS-survey}. Concerning degenerate Tur\'an numbers,
there are several general conjectures
(see \cite{FS-survey}).  First, Erd\H{o}s and Simonovits conjectured that if
$\cH$ is a finite family with $\chi(\cH)=2$ then there is a rational $r\in [1,2)$ and a constant $c>0$
such that $\lim_{n\to \infty}\ex(n,\cH)/n^r =c$. (See Conjecture 1.6 of \cite{FS-survey}).
This conjecture is still wide open. In fact the order of magnitude of $\ex(n,\cH)$ where $\chi(\cH)=2$ is known only for very few families $\cH$.
Another conjecture, which may be viewed as the inverse extremal problem of the previous one, is that
for very rational $r\in [1,2)$ there exists a finite family $\cH$ of graphs such that $c_1n^r< \ex(n,\cH)<c_2 n^r$
for some constants $c_1,c_2$. (See Conjecture 2.37 of \cite{FS-survey}.)
In a recent breakthrough work by Bukh and Conlon \cite{BC}, this second conjecture has been verified, using
a random algebraic method (developed earlier in \cite{BBK, bukh, conlon}).

However, the following analogous problem on the Tur\'an number of a single bipartite graph, raised by Erd\H{o}s and Simonovits \cite{Erdos},
on the other hand, is still wide open.

\begin{question}{\rm (\cite{Erdos})}\label{pro:ES}
Is it true that for every rational number $r$ in $(1,2)$ there exists a single bipartite graph $H_r$ such that $\ex(n,H_r)=\Theta(n^r)$?
\end{question}

We will refer to a rational $r$ for which Problem \ref{pro:ES} has an affirmative answer as a {\it Tur\'an exponent} for a single graph.
The only known Tur\'an exponents for single graphs from the literature are rational numbers of the forms
$1+\frac{1}{s}$ and $2-\frac{1}{s}$ for all integers $s\geq 2$.  Specifically it is known that $\ex(n,K_{s,t})=\Theta(n^{2-1/s})$
when $t>(s-1)!$ (by \cite{KST, KRS, ARS}). Let $\theta_{s,p}$ denote the graph obtained by taking the union of $p$ internally
disjoint paths of length $s$ between a pair of vertices. Faudree and Simnovits \cite{FS} showed that
$\ex(n,\theta_{s,p})=O(n^{1+1/s})$ for all $p\geq 2$ (see \cite{BT} for a recent improvement on the specific bound)
while Conlon \cite{conlon} showed that for every $s\geq 2$ there exists
a $p_0$ such that for all $p\geq p_0$ we have $\ex(n,\theta_{s,p})=\Omega(n^{1+1/s})$. Hence for each $s$ and sufficiently large $p$
we have $\ex(n,\theta_{s,p})=\Theta(n^{1+1/s})$.
For a more thorough introduction to degenerate Tur\'an numbers,
the reader is referred to the recent survey by F\"uredi and Simonovits \cite{FS-survey}.

Our main theorem is as follows, which in particular establishes an infinite sequence of  new Tur\'an exponents.

\begin{theorem}\label{thm:main-exponents}
For any rational number $r=2-\frac{2}{2s+1}$, where $s\geq 2$ is an integer, or $r=\frac75$,
there exists a single bipartite graph $H_r$ such that $\ex(n,H_r)=\Theta(n^r)$.
\end{theorem}

In establishing the first part of our main theorem, we establish a stronger result concerning the Tur\'an numbers of cube-like graphs, which
answers a question of Pinchasi and Sharir  \cite{PS}.  
This result may be of independent interest. To establish the second part of our main result,
we develop an asymmetric Tur\'an bound on  $\theta_{s,p}$ which may be viewed as a common generalization (in a general sense)
of  \cite{FS} and \cite{NV}, and may also be of independent interest.

To describe our results, we need some more detailed background, which we discuss over several subsections.

%%%%

\subsection{The theorem of Bukh and Conlon and a conjecture}

To describe Bukh and Conlon's results, we need some definitions.
Given a tree $T$ together with an independent set $R\subseteq V(T)$, we call $(T,R)$ a {\it rooted tree} and $R$ the {\it root set}.
Given any $S\subseteq V(T)\setminus R$, let $e(S)$ denote the number of edges of $T$ with at least one endpoint in $S$. Let $\rho_S=e(S)/|S|$.
Let $\rho_T=\rho(V(T)\setminus R)$.
We say that the rooted tree $(T,R)$ is {\it balanced} if $\rho_S\geq \rho_T$ for all $S\subseteq V(T)\setminus R$.
Given a rooted tree $(T,R)$ and a positive integer $p$,
let $\cT^p_R$ denote the family of graphs consisting of all possible union of $p$ distinct labelled copies of $T$,
each of which agree on the root set $R$. We call $\cT^p_R$ the {\it $p$th power} of $(T,R)$.
The key result of Bukh and Conlon \cite{BC} is the following
\begin{theorem} {\rm (\cite{BC})} \label{BC-lower}
For any balanced rooted tree $(T,R)$, there exists a $p_0$ such that for all $p\geq p_0$, $$\ex(n,\cT^p_R)=\Omega(n^{2-1/\rho_T}).$$
\end{theorem}
A straightforward counting argument shows that $\ex(n,\cT^p_R)=O(n^{2-1/\rho_T})$ and
thus implies that $\ex(n,\cT^p_R)=\Theta(n^{2-1/\rho_T})$ for sufficiently large $p$.
Bukh and Conlon \cite{BC} also showed that for each rational $r$ in $(1,2)$, there exists a balanced rooted tree $(T,R)$ with $\rho_T=\frac{1}{2-r}$,
thereby establishing the existence of a family $\cH_r$ with $\ex(n,\cH_r)=\Theta(n^r)$ for each rational $r\in (1,2)$.

Let $(T,R)$ be a balanced rooted tree. Let $T^p_R$ denote the unique member of $\cT^p_R$ in which the $p$ labelled copies of $T$ 
are pairwise vertex disjoint outside $R$. By Theorem \ref{BC-lower}, 
\[\ex(n,T^p_R)\geq \ex(n,\cT^p_R)= \Omega(n^{2-1/\rho_T}).\]
If there exists a matching upper bound on $\ex(n,T^p_R)$, then together with earlier discussion this would
answer Question \ref{pro:ES} in the affirmative in a very strong way. Indeed Bukh and Conlon conjectured that
a matching upper bound indeed exists.

\begin{conjecture} \label{BC-conjecture} {\rm (\cite{BC})}
If $(T,R)$ is a balanced rooted tree, then 
\[\ex(n, T^p_R) =O(n^{2-1/\rho_T}).\]
\end{conjecture}

Let $D_s$ be the tree obtained by taking two disjoint stars with $s$ leaves and joining the two central vertices by an edge,
and $R$ the set of all the leaves in $D_s$.
It is easy to check that $(D_s,R)$ is balanced with $\rho_{D_s}=\frac{2s+1}{2}$.
Let $T_{s,t}=D_s^R$. We will show that (in Corollary \ref{main-corollary}) that $\ex(n,T_{s,t})=O(n^{2-\frac{2}{2s+1}})$ for all $t\geq s\geq 2$ and
thereby verifying Conjecture \ref{BC-conjecture} for $T=D_s$, $R$ be its set of leaves, making
a first step towards Conjecture \ref{BC-conjecture}.

\subsection{The cube and its generalization}

Let $Q_8$ denote the $3$-dimensional cube, that is, the graph obtained from two vertex-disjoint $C_4$'s by adding a perfect matching
between them.
The well-known cube theorem of Erd\H{o}s and Simonovits \cite{cube} states that
\begin{equation}\label{equ:cube-upper}
\ex(n,Q_8)=O(n^{8/5}).
\end{equation}
Pinchasi and Sharir \cite{PS} gave a new proof of this and extended to the  bipartite setting.
More recently, F\"uredi \cite{furedi-cube} 
showed that $\ex(n,Q_8)\leq n^{8/5}+(2n)^{3/2}$, giving another new proof of the cube theorem.

Pinchasi and Sharir's approach is motivated by certain geometric incidence problems.
In their approach it is more convenient to view $Q_8$ as
a special case of the graph $H_{s,t}$ defined as follows.
Let $t\geq s\geq 2$ be positive integers.
Let $M$ be an $s$-matching $a_1b_1,a_2b_2,\ldots, a_sb_s$,
and $N$ a $t$-matching $c_1d_1,c_2d_2, \ldots, c_t d_t$,
where $M$ and $N$ are vertex disjoint.
Let $H_{s,t}$ be obtained from $M\cup  N$
by adding edges $a_id_j$ and $b_ic_j$ over all
$i\in [s]$ and $j\in [t]$.

Alternatively, we may view $H_{s,t}$ as being obtained from two vertex disjoint
copies of $K_{s,t}$ by adding a matching that joins the two images of every vertex in $K_{s,t}$.
In particular, we see that $Q_8=H_{2,2}$.
In addition to giving a new proof of \eqref{equ:cube-upper},
Pinchasi and Sharir \cite{PS} proved that
if $G$ is an $n$-vertex graph that contains neither a copy of
$H_{s,t}$ nor a copy of $K_{s+1,s+1}$, then $e(G)\leq O(n^{2-2/(2s+1)})$.

\begin{question} \label{PS} {\rm \cite{PS}}
Is it true that for all $t\geq s\geq 2$,
\[\ex(n,H_{s,t})=O(n^{2-\frac{2}{2s+1}})?\]
\end{question}

This was answered affirmatively if $s=t$ in \cite{JN}.

In this paper, we answer Pinchasi and Sharir's question affirmatively as follows.
\begin{theorem}\label{thm:main}
For any $t\geq s\geq 2$, $\ex(n,H_{s,t})=O\left(n^{2-2/(2s+1)}\right)$.
\end{theorem}

Note that $T_{s,t}\subseteq H_{s,t}$. Hence, 
by Theorem \ref{BC-lower} we have the following.

\begin{proposition} \label{prop-lower}
There exists a function $\ell$ such that for all $s\geq 2$ and $t\geq \ell(s)$,
$$\ex(n,H_{s,t})\geq \ex(n,T_{s,t})\geq \ex(n,\cT^t_R)\geq \Omega\left(n^{2-2/(2s+1)}\right).$$
\end{proposition}

\noindent Theorem \ref{thm:main} and Proposition \ref{prop-lower} now give

\begin{corollary} \label{main-corollary}
There exists a function $\ell$ such that for all $s\geq 2$ and $t\geq \ell(s)$,
$$\ex(n,H_{s,t})=\Theta(n^{2-2/(2s+1)}) \quad \mbox{and} \quad
\ex(n,T_{s,t}) = \Theta(n^{2-2/(2s+1)}).$$
\end{corollary}

%%%%%%%%%%%%%%%%%%%

\subsection{Theta graphs and $3$-comb-pastings}
For the second part of our work, we give another new Tur\'an exponent of $7/5$ for a bipartite graph $S_p$
which we define below.

By a {\it $3$-comb} $T_3$, we denote the tree obtained from a $3$-vertex path $P=abc$ by adding three new vertices $a',b',c'$
and three new edges $aa', bb', cc'$.
For each $p\geq 2$, a {\it $3$-comb-pasting}, denoted by $S_p$, is the graph obtained by first taking $p$ vertex disjoint
copies of $T_3$ and then combining the images of $a'$ into one vertex, the images of $b'$ into one vertex, and the images of $c'$ into one
vertex.

Let $R$ denote the set of leaves of $T_3$.
It is easy to see that $(T_3,R)$ is balanced with density $5/3$,
while the $3$-comb-pasting $S_p$ is just a member in the $p$th power of $(T_3,R)$.
Hence, by Theorem \ref{BC-lower}, there exists $p_0$ such that for all $p\geq p_0$, $\ex(n,S_p)\geq \Omega(n^{7/5})$.

We prove a matching upper bound as follows.

\begin{theorem} \label{thm:3-comb-pasting}
For all $p\geq 2$, it holds that $\ex(n,S_p)=O(n^{7/5})$.
\end{theorem}

\begin{corollary}
There exists a positive integer $p_0$ such that for all $p\geq p_0$, it holds that
\[\ex(n,S_p)=\Theta(n^{7/5}).\]
\end{corollary}

A key step in the proof of Theorem \ref{thm:3-comb-pasting} is to study Tur\'an numbers of theta graphs in the bipartite setting,
which continue the line of work of Faudree and Simonovits \cite{FS} and of Naor and Verstraete \cite{NV} and may be of independent interest.

Given a family $\cH$ of graphs and positive integers $m, n$,
the {\it asymmetric bipartite Tur\'an number} $z(m,n,\cH)$ of $\cH$ denote
the maximum number of edges in an $m$ by $n$ bipartite graph that does not contain any member of $\cH$
as a subgraph. If $\cH$ has just one member $H$, we write $z(m,n,H)$ for $z(m,n,\{H\})$.
The function $z(m,n, C_{2k})$ had been studied in the context of number theoretic problems and geometric problems.
Naor and Verstra\"ete \cite{NV} proved that for $m\le n$ and $k\ge 2$,
\begin{equation*}\label{equ:z(C2k)-0}
z(m,n,C_{2k})\le \left\{\begin{array}{ll}
    (2k-3)\cdot [(mn)^{\frac{k+1}{2k}}+m+n] & \text{if } k \text{ is odd}, \\
    (2k-3)\cdot [m^{\frac{k+2}{2k}}n^{\frac{1}{2}}+m+n] & \text{if } k \text{ is even}.
  \end{array}\right.
\end{equation*}
A different form of upper bounds on $z(m,n,C_{2k})$ can be found in \cite{JM}.

Recall the theta graph $\theta_{k,p}$, that is the graph consisting of the union of
$p$ internally disjoint paths of length $k$ joining a pair of vertices.
In particular, $\theta_{k,2}=C_{2k}$.
%The Tur\'an numbers of theta graphs were studied by Faudree and Simonovits \cite{FS} and the asymmetric bipartite Tur\'an numbers of even cycles were studied by Naor and Verstra\"ete \cite{NV}. Their results are as follows.
The following result can be viewed as a common generalization of the results in \cite{FS,NV}.
\begin{theorem}  \label{thm:asymm-theta}
Let $m,n, k,p\geq 2$ be integers, where $m\leq n$. There exists a positive constant $c=c(k,p)$ such that
\begin{equation*}
z(m,n,\theta_{k,p})\le \left\{\begin{array}{ll}
    c\cdot [(mn)^{\frac{k+1}{2k}}+m+n] & \text{if } k \text{ is odd}, \\
    c\cdot [m^{\frac{k+2}{2k}}n^{\frac{1}{2}}+m+n] & \text{if } k \text{ is even}.
  \end{array}\right.
\end{equation*}
\end{theorem}

The rest of the paper is organized as follows. In Section 2, we give some preliminaries.
In Section 3, we prove Theorem \ref{thm:main}.
In Section 4, we prove Theorem \ref{thm:asymm-theta}.
In Section 5, we prove Theorem \ref{thm:3-comb-pasting}.

%%%%%%%%%%%%%%%%%%%%%%%%%%%%%

\section{Preliminaries} \label{Preli}
In this section we present some of the auxiliary lemmas which are used in the proofs of main results. The first three are  folklore, the  proofs of the other two can also be found in \cite{JN}.

\begin{lemma}\label{prop:maxcut}
If $G$ is a graph with average degree $d$ then it contains a bipartite subgraph $G_1$ with $e(G_1)\geq \frac12 e(G)$ and a subgraph $G_2$ with minimum degree $\delta(G_2)\geq \frac{1}{2}d$.
%Thus, $G$ contains a bipartite subgraph with minimum degree at least $\frac{1}{4} d(G)$.
\end{lemma}

\begin{lemma} \label{lem:min-degree}
Let $G$ be a bipartite graph with a bipartition $(A,B)$. Let $d_A=e(G)/|A|$ and $d_B=e(G)/|B|$.
There exists a subgraph $G'$ of $G$ with $e(G')\geq \frac12 e(G)$ such that each vertex in
$V(G')\cap A$ has degree at least $\frac14 d_A$ in $G'$ and each vertex in $V(G')\cap B$ has
degree at least $\frac14 d_B$ in $G'$.
\end{lemma}

\begin{lemma} \label{tree-embed}
Let $k$ be a positive integer and $T$ be a rooted tree with $k$ edges. If $G$ is a graph with minimum degree at least $k$ and $v$ is any vertex in $G$, then $G$ contains a copy of $T$ rooted at $v$.
\end{lemma}

\begin{lemma}[\cite{JN}, Lemma 5.3] \label{matching-count}
Let $t$ be a positive integer and $G$ be an $n$-vertex bipartite graph with at least $4t n$ edges. Then the number of $t$-matchings in $G$ is at least $\frac{e(G)^t}{2^tt!}$.
\end{lemma}

\begin{lemma} [\cite{JN}, Lemma 5.5] \label{H1t-count}
Let $t$ be a positive integer and $G$ be an $n$-vertex bipartite graph with a bipartition $(A,B)$.
Suppose $G$ has at least $4\sqrt{2t} n^{3/2}$ edges. Then the number
of $H_{1,t}$'s in $G$ is at least $$\frac{1}{2^{5t+2} t!}\cdot \frac{e(G)^{3t+1}}{|A|^{2t}|B|^{2t}}.$$
\end{lemma}

We also need the following regularization theorem of Erd\H{o}s and Simonovits which is an important tool for Tur\'an-type problems of sparse graphs. Recently, the first and third author have developed a version of this result for  linear hypergraphs~\cite{JY}. For a positive real $\lambda$, $G$ is  called {\it $\lambda$-almost-regular} if $\Delta(G)\leq \lambda\cdot \delta(G)$.

\begin{theorem}{\rm (\cite{cube})} \label{almost-regular}
Let $\alpha$ be any real in $(0,1)$, $\lambda=20\cdot 2^{(1/\alpha)^2}$,
and $n$ be a sufficiently large integer depending only on $\alpha$.
Suppose $G$ is an $n$-vertex graph with $e(G)\geq n^{1+\alpha}$.
Then $G$ has a $\lambda$-almost-regular subgraph on $m$ vertices,
where $m>n^{\alpha\frac{1-\alpha}{1+\alpha}}$ such that
$e(G')>\frac{2}{5} m^{1+\alpha}$.  \qed
\end{theorem}

%%%%%%%%%%%%%%%%%%%%%%%%%%%%%

\section{Tur\'an numbers of generalized cubes}

In this section we prove Theorem \ref{thm:main}. Our proof is partly based on the ideas of Pinchasi and Sharir \cite{PS}. The key new idea is Lemma \ref{correlated-matchings}. 
To state the lemma, we need some notation.

In a graph $G$, for any $S\subseteq V(G)$,
%We use $G[S]$ to denote the subgraph of $G$ induced by $S$.
{\it the common neighbourhood} of $S$ in $G$ is defined by $N_G(S)=\bigcap_{v\in S} N_G(v),$
and the {\it common degree of $S$} in $G$ is $d_G(S)=|N_G(S)|$.
When $G$ is clear from the context, we will drop the subscripts. For a matching $M$ in the bipartite graph $G$ with bipartition $(A,B)$, we define 
$A_M=V(M)\cap A, ~~B_M= V(M)\cap B$.  We call the  subgraph induced by the vertex sets $ N(B_M)\setminus V(M) $ and  $N(A_M)\setminus V(M)$ the {\it neighbourhood graph} of $M$ and with some abuse of notation, for brevity, we denote it by $N(M)$.

Let $M$ and $L$ be two matchings in $G$. We write  $M\sim L$ if $L$ is a subgraph in $N(M)$. For $t$ non-negative integer, we say that an ordered pair $(M,L)$  of matchings is {\it $t$-correlated} if $M \sim L$ and there exists a vertex $v$ in $V(M)$ such that $d_{N(L)}(v)\geq t$.

\begin{lemma} \label{correlated-matchings}
Let $G$ be an $H_{s,t}$-free bipartite graph and $M$ be an $(s-1)$-matching in $G$.
Then the number of $s$-matchings $L$ in $N(M)$ such that $(M,L)$ is $2t$-correlated is at most $(s-1)(t-1)\cdot e(N(M))^{s-1}v(N(M))$.
\end{lemma}
\begin{proof}  It suffices to prove that for any $x\in A_M$, $L'$ an $(s-1)$-matching in $N(M)$, and $y\in (V(N(M))\cap B)\setminus V(L')$, the number of $s$-matchings $L$ in $N(M)$ that contain $L'$ and $y$ and satisfy $d_{N(L)}(x)\geq 2t$ is
at most $t-1$.

Suppose otherwise, let $L'=\{c_1d_1,\dots, c_{s-1}d_{s-1}\}$, where $c_1,\dots, c_{s-1}\in A$ and
$d_1,\dots, d_{s-1}\in B$ for whih the claim fails. Then there exist $t$ distinct $s$-matchings $L_1,\dots, L_t$ in $N(M)$ containing $L'$ and $y$
that satisfy $d_{N(L_i)}(x)\geq 2t$. Let $u_1, u_2,\dots, u_t$ be distinct vertices such that $L_i=L'\cup \{u_i y\}$. For each $i\in [t]$, since $d_{N(L_i)}(x)\geq 2t$, we have $|N_{N(L_i)}(x)\setminus B_M|\geq t$.
 We can therefore find $t$ distinct vertices $v_1,\dots, v_t$  such that for each $i\in [t]$
 $v_i\in N_{N(L_i)}(x)\setminus B_M$. 

Let $B^*=\{b_1,\dots, b_{s-1}, y\}, C^*=\{x, c_1,\dots, c_{s-1}\}, U^*=\{u_1, \dots, u_t\}$, and $V^*=\{v_1,\dots, v_t\}$. It is easy to see that $G_1:=G[B^*\cup U^*]$, $G_2:=G[C^*\cup V^*]$ are both copies of $K_{s,t}$. Let $M_1:=\{u_1v_1,\dots, u_tv_t\}$, $M_2=\{b_1c_1,\dots,b_{s-1}c_{s-1}, xy\}$. One can easily check that $G_1\cup G_2\cup M_1\cup M_2$ is a copy of $H_{s,t}$ in $G$, contradicting $G$ being $H_{s,t}$-free.

%Since there are $s-1$ choices for $x$, no more than $e(N(M))^{s-1}$ choices for $L'$ and no more than $D_M$ choices for $y$, by Claim %\ref{correlated-matchings}.1,
%the number of $s$-matchings $L$ in $N(M)$ such that $(M,L)$ is $2t$-correlated and $d_{N(L)}(x)\geq 2t$ for some vertex $x\in A_M$ is at most
%$(s-1)(t-1)e(N(M))^{s-1} D_M$. By a similar argument, the number of $s$-matchings $L$ in $N(M)$ such that $(M,L)$ is $2t$-correlated and
%$d_{N(L)}(y)\geq 2t$ for some vertex $y\in B_M$ is at most $(s-1)(t-1)e(N(M))^{s-1}C_M$. Hence, the number of $s$-matchings $L$ in $N(M)$
%such that $(M,L)$ is $2t$-correlated is at most $(s-1)(t-1) e(N(M))^{s-1}v(N(M))$.
\end{proof}

{\noindent \bf Proof of Theorem \ref{thm:main}.} Our choice of constant $C$ here will be explicit. Let $\alpha=\frac{2s-1}{2s+1}$. As $s\geq 2$, we have $\frac{3}{5}\leq \alpha<1$.
Let $\lambda$ be the constant derived from Theorem \ref{almost-regular} applied for $\alpha$.
By Theorem \ref{almost-regular}, it suffices to show that there is a constant $C=C(s,t)>0$ such that the following holds for sufficiently large $n$:
if $G$ is a $\lambda$-almost-regular graph with $n$ vertices and $m\geq Cn^{1+\alpha}$ edges,
then $G$ contains a copy of $H_{s,t}$. By Proposition \ref{prop:maxcut}, we may further assume that $G$ is bipartite with a bipartition $(A,B)$. Let $\cM$ be the collection of all $(s-1)$-matchings in $G$. Denote
\begin{align*}
\cM_1&=\{M: M \in \cM,  e(N(M))\leq 2^{s+1} s! (s-1)(t-1)  v(N(M))\},\\
\cM_2&=\cM\setminus \cM_1,\\
\cM_2^{2t,s}&=\{(M,L): M\in \cM_2, L\in \cL, L \mbox{ is an $s$-matching}, M\sim L, (M,L) \mbox { is not $2t$-correlated} \}.
\end{align*}

We suppose that $G$ is $H_{s,t}$-free and derive a contradiction on the number of edges of the graph $G$. For doing so, we will use upper and lower bounds on the size of the set $\cM_2^{2t,s}$.
\medskip

{\bf Claim \ref{thm:main}.1.} $\sum_{M\in \cM_2} e(N(M))= \Omega(\frac{m^{3s-2}}{n^{4s-4}})$.

\medskip

\noindent {\it Proof of Claim.} Let us call a tree obtained from $K_{1,p}$ by subdividing each edge once a $p$-spider
of height $2$. Note that $\sum_{M\in \cM} v(N(M))$ counts the number of $(s-1)$-spiders of height $2$ in $G$.
Since $G$ is $\lambda$-almost-regular, $\Delta:=\Delta(G)\leq \lambda\cdot \delta(G)\leq \lambda\cdot 2m/n$. Thus,
\[
\sum_{M\in \cM} v(N(M))\leq n \Delta^{2s-2}= O\left(\frac{m^{2s-2}}{n^{2s-3}}\right).
\]

\noindent By the definition of $\cM_1$, we have $\sum_{M\in \cM_1} e(N(M))= O\left(\sum_{M\in \cM_1} v(N(M))\right)= O\left(\frac{m^{2s-2}}{n^{2s-3}}\right).$

\noindent On the other hand, $\sum_{M\in \cM} e(N(M))$ counts the number of $H_{1,s-1}$'s in $G$.
So by Lemma \ref{H1t-count}, we have $
\sum_{M\in \cM} e(N(M)) =  \Omega\left( \frac{m^{3s-2}}{n^{4s-4}}\right).
$
Since $m\geq Cn^{4s/(2s+1)}$ and $n$ is sufficiently large, we have $\frac{m^{3s-2}}{n^{4s-4}}\gg \frac{m^{2s-2}}{n^{2s-3}}$ thus the claim follows. \qed

\medskip

Now consider a matching $M\in \cM_2$. By Lemma \ref{matching-count},
the number of $s$-matchings $L$ in $N(M)$ is at least $(1/2^s s!)e(N(M))^s$. By Lemma \ref{correlated-matchings} and the
definition of $\cM_2$, the number of $s$-matchings $L$ in $N(M)$ such that $(M,L)$ is $2t$-correlated is at most

\[(s-1)(t-1) e(N(M))^{s-1} v(N(M)) \leq \frac{ e(N(M))^s}{2^{s+1} s!}.\]

Hence the number of $s$-matchings $L$ in $N(M)$ such that
$(M,L)$ is not $2t$-correlated is at least $(1/2) (1/2^s s!)e(N(M))^s$.

By Claim \ref{thm:main}.1, the convexity of the function $f(x)=x^s$ and the fact that $|\cM_2|\leq m^{s-1}$,
\[
|\cM_2^{2t,s}|\geq (1/2^{s+1} s!)\sum_{M\in \cM_2} e(N(M))^s= \Omega\left(\frac{(\sum_{M\in \cM_2} e(N(M)))^s}{|\cM_2|^{s-1}}\right)
=\Omega\left(\frac{m^{2s^2-1}}{n^{4s^2-4s}}\right).
\]

\medskip

{\bf Claim \ref{thm:main}.2.} $|\cM_2^{2t,s}| \leq  \binom{t-1}{s-1}(2t-1)^{s-1}m^s$.

\medskip

{\noindent \it Proof of Claim.} Let $L$ be an $s$-matching in $G$. Since $G$ is $H_{s,t}$-free, $N(L)$ has matching number at most
$t-1$. Since $N(L)$ is bipartite, by the K\"onig-Egerv\'ary theorem it has a vertex cover $Q$ of size at most $t-1$. Let $Q^+$ denote the
set of vertices in $Q$ that have degree at least $2t$ in $N(L)$ and $Q^-=Q\setminus Q^+$. 
If $M$ is an $(s-1)$-matching in $G$
that satisfies $M\sim L$ and that $(M,L)$ is not $2t$-correlated, then
$M$ is contained in $N(L)$ and could not contain any vertex in $Q^+$.
Since $Q=Q^+\cup Q^-$ is a vertex cover in $N(L)$, each edge of $M$ must contain a vertex in $Q^-$. Thus,
\[|\cM_2^{2t,s}| \leq \binom{|Q^-|}{s-1} (2t-1)^{s-1} m^s\leq \binom{t-1}{s-1}(2t-1)^{s-1}m^s.\]
%$M$ is contained in $N(L)$ and contains a vertex of $Q^-$ (since it doesn't contain a vertex of $Q^+$ while $Q$ is a vertex cover of $N(L)$).
%Since each vertex in $Q^-$ has degree at most $2t-1$ in $N(L)$, the number of such $M$ is at most
%\[\binom{|Q^-|}{s} (2t-1)^s\leq \binom{t-1}{s}(2t-1)^s.\]
\qed

Combining the lower and upper bounds on $|\cM_2^{2t,s}|$, we get that $\frac{m^{2s^2-1}}{n^{4s^2-4s}}=O(m^s)$,
which implies that $m= O(n^{4s/(2s+1)})$, where the constant factor in $O(\cdot)$ only depends on $s$ and $t$.
This contradicts that $m\geq Cn^{4s/(2s+1)}$, assuming $C$ is chosen to be sufficiently large.
\qed

%%%%%%%%%%%%%%%%%%%%%%%%%%%%%%%%%%%%%%%%%%%%%%%%%%%%%%%%%%%%%%%%%%%%%%%%

\section{Asymmetric bipartite Tur\'an numbers of Theta graphs}
In this section we establish a upper bound (i.e., Theorem \ref{thm:asymm-theta}) of the asymmetric bipartite Tur\'an numbers of theta graphs $\theta_{k,p}$.
This, in turn, will be crucial in the proof of Theorem \ref{thm:3-comb-pasting}.

Our proof, in a conspectus, employs the standard breadth-first-search tree (BFS-tree) approach and
thus the major challenge is to show that the distance levels of the BFS-tree should grow in magnitude rapidly.
This will be essentially unravelled by the following lemma,
where we adopt a modification of the so-called ``blowup method" by Faudree and Simonovits \cite{FS}.
A similar lemma was proved in \cite{JY}.

\begin{lemma} \label{lem:asymm-theta}
Let $k,p,t$ be positive integers, where $k,p\geq 2$ and $t\leq k-1$.
Let $T$ be a tree of height $t$ rooted at a vertex $x$. Let $A$ be the set of vertices
at distance $t$ from $x$ in $T$. Let $B$ be set of vertices disjoint from $V(T)$. Let $G$ be
a bipartite graph with a bipartition $(A,B)$. If $T\cup G$ is $\theta_{k,p}$-free
then $e(G)\leq 2kt p^t\cdot (|A|+|B|)$.
\end{lemma}
\begin{proof}
We use induction on $t$. For the basis case $t=1$, let
\[B^+=\{y\in B: d_G(y)\geq pk\}\quad \mbox{and } \quad B^-=B\setminus B^+.\]
By definition, $e(G[A\cup B^-])\leq pk |B^-|$. We show that $e(G[A\cup B^+])\leq pk |A\cup B^+|$.
Suppose that is not the case. Then $G[A\cup B^+]$ has average degree at least $2pk$ and
hence (by Proposition \ref{prop:maxcut}) contains a subgraph $H$ with minimum degree at least $pk$. If $k$ is odd then let $v$ be a vertex in
$V(H)\cap A$. If $k$ is even then let $v$ be a vertex in $V(H)\cap B$.
Let $F$ denote the union of $p$ paths of length $k-2$ that share a common endpoint $u$ but are otherwise vertex disjoint,
and view $u$ as the root of the tree $F$.
By Lemma~\ref{tree-embed}, $H$ contains a copy $F'$ of $F$ which has $v$ as its root.
Let $v_1,\dots, v_p$ denote the leaves of $F'$. By our choice of $v$, we have $v_1,\dots, v_p\in V(H)\cap B^+$.
By the definition of $B^+$, $d_G(v_i)\geq pk$ for each $i\in [p]$. Hence we can find distinct vertices $w_1,\dots, w_p$
in $A$ that lie outside $V(F')$ such that $v_iw_i\in E(G)$ for each $i\in [p]$. Now, $F'\cup \{v_1w_1,\dots, v_pw_p\}
\cup \{xw_1,\dots, xw_p\}$ forms a copy of $\theta_{k,p}$ in $T\cup G$, a contradiction. Hence, we have
$e(G[A\cup B^+])\leq pk |A\cup B^+|$. Putting everything together, we have
$e(G)\leq pk (|A|+|B|)$. So the basis case holds.

For the induction step, let us consider $t\geq 2$.
Let $x_1,\dots, x_q$ denote all children of $x$ in $T$. For each $i\in [q]$,
let $T_i$ denote the subtree of $T-x$ that contains $x_i$ and let $S_i=V(T_i)\cap A$. Then $S_1,\dots, S_q$ partition $A$.
For each $u\in A$, let $P_u$ denote the unique path from $u$ to $x$ in $T$.
Then for each $u\in A$, $P_u$ has length $t$, and if $u,v$ lie in different $S_i$'s then $V(P_u)\cap V(P_v)=\{x\}$.
Let
\[B^*=\{y\in B: d_G(y)<kp^2\}.\]
By definition, it is clear that we have
\begin{equation}\label{equ:B^*}
e(G[A\cup B^*])\leq kp^2\cdot |B^*|.
\end{equation}
We further partition $B\setminus B^*$ into the following two sets.
Let
\[B^+=\{y\in B \setminus B^*: \mbox{ for } \forall i\in [q], y \mbox{ has no more than } d_G(y)/p \mbox{ neighbours in } S_i\}\]
and $B^-=B \setminus (B^*\cup B^+).$
We now observe the following property for $B^+$.

\medskip

{\bf Claim 1.} For each $y\in B^+$ and any subset $I\subseteq [q]$ of size $p-1$, we have $|N_G(y)\setminus \bigcup_{i\in I} S_i|>kp$.

\medskip

{\noindent \it Proof of Claim 1.} Since $y\in B^+$, for any $i\in I$, we have $|N_G(y)\cap S_i|<d_G(y)/p$. Hence,
\[\left|N_G(y)\setminus \bigcup_{i\in I} S_i\right|\geq \left(1-\frac{p-1}{p}\right)\cdot d_G(y)=\frac{d_G(y)}{p}\geq kp,\]
proving the claim. \qed

\medskip

We prove two more claims, which bounds $e(G[A\cup B^+])$ and $e(G[A\cup B^-])$, respectively.

\medskip

{\bf Claim 2.} $e(G[A\cup B^+])\leq kp\cdot |A\cup B^+|$.

\medskip

{\noindent \it Proof of Claim 2.} Let $F$ be the tree consisting of $p$ paths of length $k-t-1$ that share a common
endpoint $u$ but are otherwise vertex disjoint; also view $u$ as the root of $F$. So $F$ has $p(k-t-1)$ edges.
Suppose for a contradiction that $e(G[A\cup B^+])> kp\cdot |A\cup B^+|$.
Then by Proposition \ref{prop:maxcut}, $G[A\cup B^+]$ contains a subgraph $H$ with minimum degree more than $kp$.
If $k-t-1$ is odd, then let $v$ be a vertex in $V(H)\cap A$;
and if $k-t-1$ is even, then let $v$ be a vertex in $V(H)\cap B^+$. By Lemma~\ref{tree-embed},
$H$ contains a copy $F'$ of $F$ which has $v$ as its root.
Let $v_1,\dots, v_p$ denote the leaves of $F'$.
By our choice of $v$, we have $v_1,\dots, v_p\in V(H)\cap  B^+$.
By Claim 1, we can find vertices $w_1,\dots, w_p$ outside $V(F')$ such that
they all lie in different $S_i$'s and $v_1w_1, v_2w_2,\dots, v_pw_p\in E(G)$.
Indeed, for any $\ell\leq p$, suppose we have found $w_1,\dots, w_{\ell-1}$.
By Claim 1, $v_\ell$ has at least $kp$ neighbours that lie outside the $S_j$'s that contain vertices in $\{w_1,\dots, w_{\ell-1}\}$.
Among these neighbours we can find one that also does not lie in $V(F')$.
We let $w_\ell$ be such a vertex. Since $w_1,\dots, w_p$ all lie in different $S_i$'s,
the paths $P_{w_1},\dots, P_{w_p}$ pairwise intersect only in vertex $x$.
Now $F'\cup \{v_1w_1,\dots, w_pw_p\}\cup \bigcup_{i=1}^k P_{w_i}$ forms a copy of $\theta_{k,p}$ in $G$, a contradiction.
Hence we must have $e(G[A\cup B^+])\leq kp\cdot |A\cup B^+|$.\qed

\medskip

{\bf Claim 3.} $e(G[A\cup B^-])\leq 2k(t-1) p^t\cdot |A\cup B^-|$.

\medskip
{\noindent \it Proof of Claim 3.} For each $y\in B^-$ by definition there exists $i(y)\in [q]$
such that $|N_G(y)\cap S_{i(y)}|\geq d_G(y)/p$; let us fix such an $i(y)$.
We then define a subgraph $H$ obtained from $G[A\cup B^-]$ by only taking the edges from every $y\in B^-$ to $N_G(y)\cap S_{i(y)}$.
By the definition of $H$, we see that $e(H)\geq \frac{1}{p}e(G[A\cup B^-])$.
Now, for each $j\in [q]$ let $B_j=\{y\in B^-: i(y)=j\}$.
Then in fact $H$ is the vertex-disjoint union of
$H[S_1\cup B_1], H[S_2\cup B_2],\dots, H[S_q\cup B_q]$.
Let $j\in [q]$. Note that $T_j$ is tree of height $t-1$ rooted at $x_j$ and $B_j$ is the set of vertices at distance $t-1$ from $x_j$.
Also, $B_j$ is vertex disjoint from $T_j$ and $H[S_j\cup B_j]$ is a bipartite graph with a bipartition $(S_j, B_j)$.
Since $T_j\cup H[S_j\cup B_j]$ is $\theta_{k,p}$-free, by the induction hypothesis,
$e(H[S_j,B_j])\leq 2k(t-1)p^{t-1}\cdot |S_j\cup B_j|$. Hence,
\[e(H)=\sum_{j=1}^p e(H[S_j, B_j])\leq 2k (t-1) p^{t-1}\cdot \sum_{j=1}^p |S_j\cup B_j|\leq 2k (t-1) p^{t-1}\cdot |A\cup B^-|,\]
implying that
$e(G[A\cup B^-])\leq p\cdot e(H)\leq 2k(t-1)p^t\cdot |A\cup B^-|.$
This proves Claim 3.\qed

\bigskip

Finally, combining \eqref{equ:B^*} with Claims 2 and 3, we have that
\begin{eqnarray*}
e(G) &=&e(G[A\cup B^*])+e(G[A,B^+])+e(G[A,B^-])\\
&\leq & kp^2\cdot |B^*| + kp\cdot (|A|+|B^+|)+ 2k(t-1) p^t\cdot (|A|+|B^-|)< 2ktp^t\cdot (|A|+|B|),
\end{eqnarray*}
finishing the proof of Lemma \ref{lem:asymm-theta}.
\end{proof}

We are ready to show Theorem \ref{thm:asymm-theta}.

\medskip

{\noindent \bf Proof of Theorem \ref{thm:asymm-theta}:}
Let $G$ be a $\theta_{k,p}$-free bipartite graph with a bipartition $(A,B)$ where $|A|=m$ and $|B|=n$.
Let $c=16k^2p^k$.
If $k$ is odd, then we assume $e(G)\geq c\cdot (mn)^{\frac{1}{2}+\frac{1}{2k}}+c\cdot (m+n)$.
If $k$ is even, then assume $e(G)\geq c\cdot m^{\frac{1}{2}+\frac{1}{k}} n^{\frac{1}{2}} +c\cdot (m+n)$.
Let $d_A=e(G)/|A|$ and $d_B=e(G)/|B|$. So each of $d_A, d_B$ is more than $c=16k^2p^k$.

By Lemma \ref{lem:min-degree},
$G$ contains a subgraph $G'$ with $e(G')\geq \frac{1}{2} e(G)$ such that each vertex in $V(G')\cap A$
has degree at least $\frac14d_A$ in $G'$ and that each vertex in $V(G')\cap B$ has
degree at least $\frac14 d_B$ in $G'$.
Fix a vertex $x\in V(G')\cap A$.
For each integer $i\geq 0$, let $L_i$ denote the set of vertices at distance $i$ from $x$ in $G'$,
and let $d_i=d_A$ if $i$ is odd and $d_i=d_B$ if $i$ is even.
So we see that every vertex in $L_{i-1}$ has degree at least $\frac14d_i$ in $G'$.

Using Lemma \ref{lem:asymm-theta}, we show that the (growth) ratio of two consecutive levels must be large in the following

\medskip

{\bf Claim.} For each $i\in [k]$, we have $|L_i|/|L_{i-1}|\geq \frac{d_{i}} {16kip^i}$. In particular,
$|L_i|\geq |L_{i-1}|$ holds.

\medskip

{\noindent \it Proof of Claim.} Since $d_i\geq 16k^2p^k$ for each $i$,
we observe that the second statement follows easily by the first statement.
So it suffices to prove the first statement, which we will prove by induction on $i$.
If $i=1$, then we have $\frac{|L_1|}{|L_0|}\geq \frac{1}{4} d_A\geq \frac{d_1}{16kp}$.
%Consider $i=2$. Since $G'$ is $\theta_{k,p}$-free, by Lemma \ref{lem:asymm-theta}, $e(G'[L_1\cup L_2])\leq 2kp\cdot (|L_1|+|L_2|)$.
%On the other hand, each vertex $v$ in $L_1\subseteq B$ has degree at least $(1/4)d_B$ in $G'$,
%and $N_{G'}(v)\subseteq \{x\} \cup L_2$. Hence each $v$ in $L_1$ has degree
%at least $(1/4)d_B-1\geq (1/8)d_B$ in $G'[L_1\cup L_2]$.
%So $e(G'[L_1\cup L_2])\geq (1/8)d_B |L_1|$.
%Combining, we have
%\[(1/8)d_B|L_1|\leq 2kp (|L_1|+|L_2|),\]
%which implies that
%\[|L_2|\geq \left(\frac{d_B}{16kp}-1\right)\cdot |L_1|\geq \left(\frac{d_2}{16kp^2}-1\right)\cdot |L_1|.\]
So the claim holds for the basis step.

For the inductive step, consider $i\geq 2$. Let $T_{i-1}$ be a breadth-first-search tree
in $G'$ rooted at $x$ with vertex set $L_0\cup L_1\cup \dots \cup L_{i-1}$. Applying Lemma \ref{lem:asymm-theta} to $T_{i-1}$
and $G'[L_{i-1}\cup L_i]$, we get
\[e(G'[L_{i-1}, L_i])\leq  2k(i-1) p^{i-1} (|L_{i-1}|+|L_i|).\]
Similarly, it holds that
\[e(G'[L_{i-2}\cup L_{i-1}|)\leq 2k(i-2) p^{i-2}(|L_{i-2}|+|L_{i-1}|)\leq 4k (i-2) p^{i-2}|L_{i-1}|,\]
where the last step holds because $|L_{i-1}|\geq |L_{i-2}|$ by the induction hypothesis.
All edges in $G'[L_{i-2}\cup L_{i-1}\cup L_i]$ are either in $(L_{i-2},L_{i-1})$ or in $(L_{i-1},L_i)$,
so we get that
\[e(G'[L_{i-2}\cup L_{i-1}\cup L_i])=e(G'[L_{i-2}\cup L_{i-1}|)+e(G'[L_{i-1}, L_i])\leq 2kip^i\cdot (|L_{i-1}|+|L_i|).\]
On the other hand, each vertex in $L_{i-1}$ has degree at least $\frac14 d_i$ in $G'$,
and all edges of $G'$ incident to $L_{i-1}$ lie in $G'[L_{i-2}\cup L_{i-1}\cup L_i]$.
Hence, we have
\[\frac14 d_i\cdot |L_{i-1}| \leq e(G'[L_{i-2}\cup L_{i-1}\cup L_i])\leq 2kip^i\cdot (|L_{i-1}|+|L_i|).\]
Solving for $|L_i|$, we get
$|L_i|\geq \left(\frac{d_i}{8kip^i} -1\right)\cdot|L_{i-1}|\geq \frac{d_i}{16kip^i}\cdot|L_{i-1}|,$
proving the claim.
\qed

\medskip

By the claim, we have
\[|L_k|\geq \alpha\cdot \prod_{i=1}^k d_i\cdot |L_0|=\alpha\cdot \prod_{i=1}^k d_i,\]
where $\alpha=\prod_{i=1}^k \frac{1}{16kip^i}$. Recall that $c=16k^2 p^k$.
So $\alpha c^k>1$.
Suppose first that $k$ is odd, say $k=2s+1$.
Then it follows that $L_k\subseteq B$ and
\[|L_k|\geq \alpha \cdot d_A^{s+1}d_B^s =\alpha\cdot \frac{e(G)^k}{m^{s+1}n^s}. \]
By the assumption, we have $e(G)>c\cdot (mn)^{\frac{1}{2}+\frac{1}{2k}}$,
which shows that $|L_{k}|\geq \alpha c^k n>n.$
This is a contradiction, since $L_{k}\subseteq B$ and $|B|=n$.
Now consider that $k$ is even, say $k=2s$.
Then we have
\[|L_{k}|\geq \alpha \cdot d_A^s d_B^s =\alpha\cdot \frac{e(G)^{k}}{m^s n^s}. \]
In this case $e(G)> c\cdot m^{\frac{1}{2}+\frac{1}{k}}n^{\frac{1}{2}}$.
This gives that $|L_{k}|\geq \alpha c^k\cdot \left(m^{\frac{1}{2}+\frac{1}{k}}n^{\frac{1}{2}}\right)^k/m^sn^s
=\alpha c^k\cdot m>m,$
again a contradiction, since $L_{k}\subseteq A$ and $|A|=m$.
This completes the proof of Theorem \ref{thm:asymm-theta}.
\qed

\medskip

One can promptly derive the following special case of Theorem \ref{thm:asymm-theta},
which will play an important role in the proof of Theorem \ref{thm:3-comb-pasting}.

\begin{corollary} \label{cor:theta3p}
Let $m,n\geq 2$ be integers. Then it holds that
\[z(m,n, \theta_{3,p})\leq 144p^3\cdot \left((mn)^{2/3}+m+n\right).\]
\end{corollary}

\section{The Tur\'an exponent of $7/5$}
Here we prove the existence of the Tur\'an exponent of $7/5$.
This is achieved by the combination of Theorem \ref{thm:3-comb-pasting},
which states that $\ex(n,S_p)=O(n^{7/5})$ for all $p\geq 2$,
and the matched lower bound of this function for sufficiently large $p$ from \cite{BC}.

By considering a supergraph\footnote{i.e., a graph containing $S_p$ as its subgraph.} of $S_p$, in fact we will prove a slightly stronger result than Theorem \ref{thm:3-comb-pasting}.
We start with a definition introduced by Faudree and Simonovits \cite{FS}.
Let $H$ be a bipartite graph with an ordered pair $(A,B)$ of partite sets and $t\geq 2$ be an integer.
Define $L_t(H)$ to be the graph obtained from $H$ by adding a new vertex $u$ and joining $u$ to all vertices of $A$ by internally disjoint paths of length $t-1$ such that the vertices of these paths are disjoint from $V(H)$.

We observe that the theta graph $\theta_{3,p}$ are symmetric between its two partite sets.
So $L_3(\theta_{3,p})$ is uniquely defined.
The following proposition can be verified easily.

\begin{proposition} \label{containment}
For each $p\geq 2$, we have $S_p\subseteq L_3(\theta_{3,p})$ and thus $\ex(n, S_p)\leq \ex(n,L_3(\theta_{3,p})).$
%where $S_p$ the $p$-pasting of $3$-combs, defined in Definition \ref{3-comb-pasting-def}.
\end{proposition}
\noindent Note that as a special case, the graph $L_3(\theta_{3,2})$ also denotes the subdivision of $K_4$,
where each edge of $K_4$ is replaced by an internally disjoint path of length two.

We are now in a position to prove the following strengthening of Theorem \ref{thm:3-comb-pasting}.

\begin{theorem} \label{thm:L3-theta}
For each $p\geq 2$, there exists a positive constant $c_p$ such that
\[\ex(n, L_3(\theta_{3,p}))\leq c_p n^{7/5}.\]
\end{theorem}

\begin{proof}
We will show that it suffices to choose $c_p=12^4p^6.$
Suppose for a contradiction that there exists an $n$-vertex $L_3(\theta_{3,p})$-free graph $G$ with $e(G)> c_p n^{7/5}$.
By Proposition \ref{prop:maxcut}, $G$ contains a bipartite subgraph $G_1$ with
\begin{equation} \label{G2-min-deg}
d:=\delta(G_1)\geq d(G)/4\geq (c_p/2)\cdot n^{2/5}>(4\cdot 12^3p^6)\cdot n^{2/5}.
\end{equation}

Let $x$ be a vertex of minimum degree in $G_1$. For each $i\geq 0$, let $L_i$ denote the set of vertices at distance $i$
from $x$ in $G_1$. Then $|L_1|=|\delta(G_1)|= d$.
Let $L_2^+$ denote the set of vertices $v$ in $L_2$ such that $|N_{G_1}(v)\cap L_1|\geq 2p+2$,
and $L_2^-=L_2\setminus L_2^+$.

\medskip

{\bf Claim 1.} $G_1[L_1\cup L_2^+]$ is $\theta_{3,p}$-free.

\medskip

{\noindent \it Proof of Claim 1.}  Suppose for contradiction that $G_1[L_1\cup L_2^+]$ contains a copy $F$
of $\theta_{3,p}$. Let $A,B$ denote the two partite sets of $F$ where $A\subseteq L_1$ and
$B\subseteq L_2^+$. Then $|A|=|B|=p+1$. Suppose $B=\{b_1,\dots, b_{p+1}\}$. Since each vertex
in $L_2^+$ has at least $2p+2$ neighbours in $L_1$, we can find distinct vertices $c_1,\dots, c_{p+1}$
in $L_1\setminus A$ such that $b_1c_1,\dots, b_{p+1}c_{p+1}\in E(G_1)$. Now $F$ together with the paths
$b_1c_1x, \dots, b_{p+1}c_{p+1} x$ form a copy of $L_3(\theta_{3,p})$ in $G$, a contradiction. \qed

\medskip

{\bf Claim 2.} $|L_2|\geq d^2/(24^3p^{9/2})$.

\medskip
{\noindent \it Proof of Claim 2.}
By Claim 1 and Corollary \ref{cor:theta3p}, we have
\[ e(G_1[L_1\cup L_2^+])\leq 144 p^3\cdot \left(|L_1|^{2/3}|L_2^+|^{2/3}+|L_1|+|L_2^+|\right); \]
and by the definition of $L_2^-$,
$e(G_1[L_1\cup L_2^-])\leq (2p+2)\cdot |L_2^-|.$
Adding these inequalities up, we have
\begin{equation} \label{L1L2-upper}
e(G_1[L_1,L_2])=e(G_1[L_1\cup L_2^+])+e(G_1[L_1\cup L_2^-])\leq 144 p^3\cdot \left(|L_1|^{2/3}|L_2|^{2/3}+|L_1|+|L_2|\right).
\end{equation}
Since every vertex in $L_1$ has at least $d-1\geq 3d/4$ neighbours in $L_2$, it follows that
\begin{equation*} \label{L1L2-lower}
(3d/4)|L_1|\leq e(G_1[L_1,L_2])\leq 144 p^3\cdot \left(|L_1|^{2/3}|L_2|^{2/3}+|L_1|+|L_2|\right).
\end{equation*}
Since $d\geq 4\cdot 12^3p^6$, we see $144p^3|L_1|\leq (d/4)|L_1|$.
Thus it follows that either $144 p^3|L_1|^{2/3}|L_2|^{2/3}\geq (d/4)|L_1|$ or $144 p^3 |L_2|\geq (d/4)|L_1|$.
Using $|L_1|=d$, we get that
\[|L_2|\geq \min\left\{\frac{d^2}{24^3p^{9/2}}, \frac{d^2}{24^2 p^3}\right\}=\frac{d^2}{24^3p^{9/2}},\]
proving this claim.
\qed

\medskip

Next we consider the subgraph $H$ of $G_1$ induced on $L_2\cup L_3$, i.e.,
\[H=G_1[L_2\cup L_3].\]
Our goal in the rest of the proof is to reach a contradiction
by showing that $H$ can not contain theta graphs $\theta_{3,s}$ for large $s$,
which in turn shows that $|L_3|$ must be $\Omega(d^{5/2})$
and thus exceed the total number of vertices in $G$.

Let $T$ be a breadth-first search tree rooted at $x$ with vertex set $\{x\}\cup L_1\cup L_2$.
Let $x_1,\dots, x_m$ be the children of $x$ in $T$. For each $i\in [m]$, let $S_i$ be the set of
children of $x_i$ in $T$. Then $S_1,\dots, S_m$ partition $L_2$.
Since each vertex in $L_2$ has degree at least $d$ in $G_1$,
we have
\[e(G_1[L_1\cup L_2])+e(G_1[L_2\cup L_3])\geq d|L_2|.\]
On the other hand, by \eqref{G2-min-deg} and Claim 2, we have $d\geq 4\cdot 12^3p^6$
and thus $(d|L_2|)^{1/3}\geq 4\cdot 144\cdot p^3$,
which together with \eqref{L1L2-upper} imply that
\[e(G_1[L_1\cup L_2])\leq d|L_2|/4+144p^3(|L_1|+|L_2|)\leq d|L_2|/2.\]
Hence
\begin{equation}\label{H-lower}
e(H)=e(G_1[L_2\cup L_3])\geq d|L_2|/2.
\end{equation}

Given a vertex $u\in L_3$ and some $S_i$, we say the pair $(u,S_i)$ is {\it rich}, if $u$ has at least $2p+1$ neighbours of $H$ in $S_i$.
Let $E_H(u,S_i)$ denote the set of all edges in $H$ between $u$ and $S_i$.
We now partition $H$ into two (spanning) subgraphs $H_1,H_2$ such that
\[E(H_1)=\bigcup E_H(u,S_i) \quad \mbox{ and }\quad E(H_2)=E(H)\setminus E(H_1),\]
where the union in $E(H_1)$ is over all rich pairs $(u,S_i)$.
Note that by this definition, any $u\in L_3$ has at most $2p$ neighbours of $H_2$ in any $S_i$, i.e., $|E_{H_2}(u,S_i)|\leq 2p$.
Let $H_3$ be a subgraph of $H_2$ obtained by including exactly one edge in $E_{H_2}(u,S_i)$ over all pairs $(u,S_i)$ with $|E_{H_2}(u,S_i)|\geq 1$.
By the above discussion, it follows that
\begin{equation} \label{H3-bound}
e(H_3)\geq e(H_2)/(2p),
\end{equation}
and for any $u\in L_3$, all its neighbours in $H_3$ belong to distinct $S_i$'s.

\medskip

{\bf Claim 3.} $H_1$ is $\theta_{3,p^2}$-free.

\medskip

{\noindent \it Proof of Claim 3.} Suppose for contradiction that $H_1$ contains a copy $F$ of $\theta_{3,p^2}$.
Suppose $F$ consists of $p^2$ internally disjoint paths of length three between $u$ and $v$ where $u\in L_3$ and $v\in L_2$.
Let these paths be $ua_1b_1v, ua_2b_2v,\dots, ua_{p^2}b_{p^2}v$, where $a_1,\dots, a_{p^2}\in L_2$ and $b_1,\dots, b_{p^2}\in L_3$.

We consider two cases.
First, suppose that there exists some $S_i$ which contains $p$ different $a_j$'s.
Without loss of generality, suppose that $S_1$ contains $a_1,\dots a_p$.
For each $j\in [p]$, since $b_ja_j\in E(H_1)$,
by definition $(b_j,S_1)$ is a rich pair, i.e., there are at least $2p+1$ edges of $H$ from $b_j$ to $S_1$.
Similarly as $ua_1\in E(H_1)$, there are at least $2p+1$ edges of $H$ from $u$ to $S_1$.
Hence we can find distinct vertices $u', a_1',\dots, a_p'\in S_1\setminus \{a_1,\dots, a_p\}$
such that $uu', a_1'b_1, \dots, a'_pb_p\in E(H)$. Now
$F\cup \{uu', a'_1b_1,\dots, a'_pb_p\}\cup \{x_1u', x_1a_1',\dots, x_1a_p'\}$ forms a copy of $L_3(\theta_{3,p})$ in $G$, a contradiction.

Next, suppose that each $S_i$ contains at most $p-1$ different $a_j$'s. Then among $a_1,\dots, a_{p^2}$ we can find $p+1$ of them,
say $a_1,\dots, a_{p+1}$ that all lie in different $S_i$'s. Furthermore, we may assume that $a_1,\dots, a_p$ are outside the $S_i$'s that
contains $v$. Now $F$ together with the paths in $T$ from $x$ to $a_1,\dots, a_p,v$ form a copy of $L_3(\theta_{3,p})$ in $G$, a contradiction.
Hence $H_1$ must be $\theta_{3,p^2}$-free.
\qed

\medskip

{\bf Claim 4.} $H_3$ is $\theta_{3,p}$-free.

\medskip

{\noindent \it Proof of Claim 4.} Suppose for contradiction that $H_3$ contains a copy $F$ of $\theta_{3,p}$.
Suppose $F$ consists of $p$ internally disjoint paths of length three between $u$ and $v$,
where $u\in L_3$ and $v\in L_2$. Suppose these paths are $ua_1b_1v,\dots, ua_pb_pv$,
where $a_1,\dots, a_p\in L_2$ and $b_1,\dots, b_p\in L_3$.
By the definition of $H_3$, since $ua_1,\dots, ua_p\in E(H_3)$, $a_1,\dots, a_p$ must all lie in different $S_i$'s.
Also, for each $j\in [p]$ since $b_ja_j, b_jv\in E(H_3)$, $a_j$ and $v$
must lie in different $S_i$. So $a_1,\dots, a_p$ and $v$ all lie in different $S_i$'s.
Now, $F$ together with the paths in $T$ from $x$ to $a_1,\dots, a_p,v$ respectively
form a copy of $L_3(\theta_{3,p})$ in $G$, a contradiction. \qed

\medskip

Now, we consider two cases.

\medskip

{\bf Case 1.} $e(H_1)\geq e(H)/2$. In this case, by \eqref{H-lower}, we have
$e(H_1)\geq d|L_2|/4$. On the other hand, by Claim 3, we see that $H_1$ is $\theta_{3,p^2}$-free,
so by Corollary \ref{cor:theta3p}, we have
\begin{equation}\label{equ:Lbound}
d|L_2|/4\leq e(H_1)\leq 144p^6\cdot\left(|L_2|^{2/3} |L_3|^{2/3}+|L_2|+|L_3|\right).
\end{equation}
Since $144p^6|L_2|\leq d|L_2|/12$,
we have either
\begin{equation*}
144p^6 |L_2|^{2/3} |L_3|^{2/3}\geq d|L_2|/12 \quad \mbox{ or }\quad 144p^6|L_3|\geq d|L_2|/12.
\end{equation*}
Using this and Claim 2 that $|L_2|\geq d^2/(24^3p^{9/2})$, we can get
\[|L_3|\geq \min\left\{\frac{d^{3/2}|L_2|^{1/2}}{12^{9/2}p^9}, \frac{d|L_2|}{12^3p^6}\right\}
=\frac{d^{3/2}|L_2|^{1/2}}{12^{9/2}p^9}\geq \frac{d^{5/2}}{2^{3/2}12^6p^{45/4}}\]
Since $d\geq (4\cdot 12^3p^6)\cdot n^{2/5}$, this yields $|L_3|>n$,
a contradiction.

\medskip

{\bf Case 2.} $e(H_2)\geq e(H)/2$. Then by \eqref{H-lower} and \eqref{H3-bound},
we have $e(H_3)\geq e(H)/4p\geq d|L_2|/8p$.
By Claim 4, $H_3$ is $\theta_{3,p}$-free. Thus, by Corollary \ref{cor:theta3p} we get
\[d|L_2|/8p\leq e(H_3)\leq 144 p^3\cdot \left(|L_2|^{2/3} |L_3|^{2/3}+|L_2|+|L_3|\right).\]
Since $p\geq 2$, the above inequality would also imply \eqref{equ:Lbound}.
So we can apply the same analysis as in Case 1 to get a contradiction.

\medskip

This completes the proof of Theorem \ref{thm:L3-theta} (and thus of Theorem \ref{thm:3-comb-pasting}).
\end{proof}

We proved in Theorem \ref{thm:L3-theta} that $\ex(n, L_3(\theta_{3,p}))\leq O(n^{7/5})$.
An important idea in this proof is to use the asymmetric bipartite Tur\'an number of $\theta_{3,p}$,
which help showing that the BFS-tree grow rapidly. The use of asymmetric bipartite Tur\'an numbers
may find applications in other Tur\'an type extremal problems.

%%%%%%%%%%%%%%%%%%%%%%%%%%%%%%%%%%%%%%%%%%%%%%%%%%

\end{document}